\documentclass[11pt]{article}
\usepackage{amsmath,amssymb,amsfonts,amsthm}
\usepackage{float}
\numberwithin{equation}{section}
\newtheorem{theorem}{Theorem}[section]
\newtheorem{definition}{Definition}[section]
\newtheorem{lemma}[theorem]{Lemma}
\newtheorem{proposition}[theorem]{Proposition}
\newtheorem{corollary}[theorem]{Corollary}

\begin{document}
\begin{center}
{\Large{\textbf{Competition Graphs of Jaco Graphs and the Introduction of the Grog Number of a Simple Connected Graph}}} 
\end{center}
\vspace{0.5cm}
\large{\centerline{(Johan Kok, Susanth C, Sunny Joseph Kalayathankal)\footnote {\textbf {Affiliation of author:}\\
\noindent Johan Kok (Tshwane Metropolitan Police Department), City of Tshwane, Republic of South Africa\\
e-mail: kokkiek2@tshwane.gov.za\\ \\
\noindent Susanth C (Department of Mathematics, Vidya Academy of Science and Technology), Thalakkottukara, Thrissur-680501, Republic of India\\
e-mail: susanth\_c@yahoo.com\\ \\
\noindent Sunny Joseph Kalayathankal (Department of Mathematics, Kuriakose Elias College), Mannanam, Kottayam- 686561, Kerala, Republic of India\\
e-mail: sunnyjoseph2014@yahoo.com}}
\vspace{0.5cm}
\begin{abstract}
\noindent Let $G^\rightarrow$ be a simple connected directed graph on $n \geq 2$ vertices and let $V^*$ be a non-empty subset of $V(G^\rightarrow)$ and denote the undirected subgraph induced by $V^*$ by, $\langle V^* \rangle$. We show that the \emph{competition graph} of the Jaco graph $J_n(1), n \in \Bbb N, n \geq 5,$  denoted by $C(J_n(1))$ is given by:\\ \\
$C(J_n(1)) = \langle V^* \rangle_{V^* = \{v_i| 3 \leq i \leq n-1\}} -\{v_i v_{m_i}| m_i = i + d^+_{J_n(1)}(v_i), 3 \leq i \leq n-2\} \cup \{v_1, v_2, v_n\}.$\\ \\
\noindent Further to the above, the concept of the \emph{grog number} $g(G^\rightarrow)$ of a simple connected directed graph $G^\rightarrow$ on $n \geq 2$ vertices as well as the general \emph{grog number} of underlying graph $G$, will be introduced. The \emph{grog number} measures the efficiency of an \emph{optimal predator-prey strategy} if the simple directed graph models a ecological predator-prey web.\\ \\
We also pose four open problems for exploratory research.
\end{abstract}
\noindent {\footnotesize \textbf{AMS Classification Numbers:} 05C07, 05C12, 05C20, 05C38, 05C70} 
\section{Introduction}
\noindent  For a general reference to notation and concepts of graph theory see [1]. For ease of self-containess we shall briefly introduce the concept of a \emph{competition graph}.
\subsection{The Competition Graph of a simple connected directed graph $G^\rightarrow$}
\noindent The concept of the competition graph $C(G^\rightarrow)$ of a simple connected directed graph $G^\rightarrow$ on $n \geq 2$ vertices, was introduced by Joel Cohen in 1968 [2]. Much research has followed and recommended reading can be found in [4, 5, 6 \emph{together with all their references}]. The concept of competition graphs found application in amongst others, Coding theory, Channel allocation in communication, Information transmission, Complex systems modelled in energy and economic applications, Decionmaking based mainly on opinion influences and Predator-Prey dynamical systems. \\ \\
\noindent For a simple connected directed graph $G^\rightarrow$ with vertex set $V(G^\rightarrow)$ the competition graph $C(G^\rightarrow)$ is the simple graph \emph{(undirected and possibly disconnected)} having $V(C(G^\rightarrow)) = V(G^\rightarrow)$ and the edges $E(C(G^\rightarrow)) = \{vy|$ \emph{if at least one vertex} $w \in V(G^\rightarrow)$ \emph{exists such that the arcs} $(v, w), (y, w)$ \emph{exist}$\}.$\\ \\
Let $G^\rightarrow$ be a simple connected directed graph and let $V^*$ be a non-empty subset of $V(G^\rightarrow)$ and denote the undirected subgraph induced by $V^*$ by, $\langle V^* \rangle$.
\subsection{The Competition graph of the Jaco graph, $J_n(1), n \in \Bbb N$}
For ease of reference the definition and basic properties of the Jaco graph $J_n(1), n \in \Bbb N$ will be repeated.\\ \\
The infinite Jaco graph (\emph{order 1}) was introduced in [3], and defined by $V(J_\infty(1)) = \{v_i| i \in \Bbb N\}$, $E(J_\infty(1)) \subseteq \{(v_i, v_j)| i, j \in \Bbb N, i< j\}$ and $(v_i,v_ j) \in E(J_\infty(1))$ if and only if $2i - d^-(v_i) \geq j.$\\ \\ The graph has four fundamental properties which are; $V(J_\infty(1)) = \{v_i|i \in \Bbb N\}$ and, if $v_j$ is the head of an edge (arc) then the tail is always a vertex $v_i, i<j$ and, if $v_k,$ for smallest $k \in \Bbb N$ is a tail vertex then all vertices $v_ \ell, k< \ell<j$ are tails of arcs to $v_j$ and finally, the degree of vertex $k$ is $d(v_k) = k.$ The family of finite directed graphs are those limited to $n \in \Bbb N$ vertices by lobbing off all vertices (and edges arcing to vertices) $v_t, t > n.$ Hence, trivially we have $d(v_i) \leq i$ for $i \in \Bbb N.$
\noindent It is important to note that the general definition of a finite Jaco graph [3], prescribes a well-defined orientation. So we have one well-defined orientation of the $2^{\epsilon(J_n(1))}$ possible orientations.
\begin{theorem}
For the Jaco graph $J_n(1), n \in \Bbb N, n \geq 5,$  the competition graph $C(J_n(1))$ is given by:\\ \\
$C(J_n(1)) = \langle V^* \rangle_{V^* = \{v_i| 3 \leq i \leq n-1\}} -\{v_iv_{m_i}| m_i = i + d^+_{J_n(1)}(v_i), 3 \leq i \leq n-2\} \cup \{v_1, v_2, v_n\}.$
\end{theorem}
\begin{proof}
The well-defined orientation of Jaco graphs renders the competition graphs of $J_n(1)_{1 \leq n \leq 4}$ to be isolated vertices only. \\ \\
From the definition of a Jaco graph it follows clearly that $J_5(1)$ is the smallest Jaco graph for which the vertex $v_5$ exists such that arcs $(v_3, v_5)$ and $(v_4, v_5)$ exist to allow the edge $v_3v_4$ in the \emph{competition graph}. It follows easily that vertices $v_1, v_2$ are isolated vertices in $C(J_5(1))$.\\ \\
Since $d^+_{J_5(1)}(v_5) = 0$, vertex $v_5$ remains isolated in $C(J_5(1))$. Hence, the edges \emph{(previously arcs respectively)}, $v_3v_5$ and $v_4v_5$ do not exist in $C(J_5(1))$. It implies that the edge \emph{(previously an arc)}, $v_3v_5 = v_3v_{m_3}, m_3= 3 + d^+_{J_5(1)}(v_3), 3 \leq 3 \leq 3 = 5 - 2$ does not exist.  So the result:\\ \\
$C(J_n(1)) = \langle V^* \rangle_{V^* = \{v_i| 3 \leq i \leq n-1\}} -\{v_i v_{m_i}| m_i = i + d^+_{J_n(1)}(v_i), 3 \leq i \leq n-2\} \cup \{v_1, v_2, v_n\},$ holds for $n = 5.$\\ \\
We will now apply induction. Assume the result holds for $n = k$ hence, assume:\\ \\
$C(J_k(1)) = \langle V^* \rangle_{V^* = \{v_i| 3 \leq i \leq k-1\}} -\{v_iv_{m_i}| m_i = i + d^+_{J_k(1)}(v_i), 3 \leq i \leq k-2\} \cup \{v_1, v_2, v_k\}.$ Also assume $J_k(1)$ has Jaconian vertex $v_i$.\\ \\
Consider $n = k+1$. This extension adds the vertex $v_{k+1}$ and the set of arcs, $\{(v_{i+1},v_{k+1}), (v_{i+2},v_{k+1}),\\ (v_{i+3},v_{k+1}), ..., (v_k,v_{k+1})\}$ to $J_k(1)$ to obtain $J_{k+1}(1).$ From the definition of the \emph{competition graph} it follows that edges $v_{i+1}v_k, v_{i+2}v_k, ..., v_{k-1}v_k$ are defined in $C(J_{k+1}(1))$. Since, $k-1 = (k+1) -2$ we have the edges $v_{i+1}v_k, v_{i+2}v_k, ..., v_{(k+1)-2}v_k$ defined in $C(J_{k+1}(1))$.\\ \\
Clearly vertices $v_1$ and $v_2$ remain isolated vertices in $C(J_{k+1}(1))$. Equally evident is that $d^+_{J_{k+1}(1)}(v_{k+1}) = 0$, so $v_{k+1}$ is an isolated vertex in $C(J_{k+1}(1))$. Hence the result:\\ \\
$C(J_{k+1}(1)) = \langle V^* \rangle_{V^* = \{v_i| 3 \leq i \leq (k+1)-1\}} -\{v_i, v_{m_i}| m_i = i + d^+_{J_{k+1}(1)}(v_i), 3 \leq i \leq (k+1)-2\} \cup \{v_1, v_2, v_{k+1}\},$ holds.\\ \\
Thus the result:\\ \\
$C(J_n(1))_{n \geq 5} = \langle V^* \rangle_{V^* = \{v_i| 3 \leq i \leq n-1\}} -\{v_iv_{m_i}| m_i = i + d^+_{J_n(1)}(v_i), 3 \leq i \leq n-2\} \cup \{v_1, v_2, v_n\}$\\ \\
is settled through induction.
\end{proof}
\section{Grog Numbers of Simple Connected Directed Graphs}
For a simple connected graph $G$ on $n \geq 2$ vertices we consider any orientation $G^\rightarrow$ thereof. Label the vertices randomly $v_1, v_2, v_3, ..., v_n$. The aforesaid vertex labelling is called is called \emph{indicing} and a specific labelling pattern is called an \emph{indice} of $G$. Consider the graph to represent a predator-prey web. A vertex $v$ with $d^+_{G^\rightarrow}(v) = 0$ is exclusively \emph{prey.} To the contrary a vertex $w$ with $d^-_{G^\rightarrow}(w) = 0$ is exclusively \emph{predator.} A vertex $z$ with $d_{G^\rightarrow}(z) = d^+_{G^\rightarrow}(z)_{>0} + d^-_{G^\rightarrow}(z)_{>0}$ is a mix of \emph{predator-prey.}\\ \\
Let a vertex  labelled $v_i$ have an initial \emph{predator$_{\geq 0}$-prey$_{\geq 0}$ population} of exactly $\rho(v_i) = i.$ So generally there is no necessary relationship between the initial \emph{predator$_{\geq 0}$-prey$_{\geq 0}$ population} $\rho(v_i) = i$ and $d_{G^\rightarrow}(v_i) = d^+_{G^\rightarrow}(v_i)_{\geq 0} + d^-_{G^\rightarrow}(v_i)_{\geq 0}.$
\subsection{The Grog algorithm}
The \emph{predator-prey} dynamics now follow the following rules.\\ \\
\textbf{Grog algorithm:}\footnote{Admittedly, the Grog algorithm has been described informally. See Open problem 3.}\\
\noindent 0. Consider the initial graph $G^\rightarrow$.\\
\noindent 1. Choose any vertex $v_i$ and predator along any number $1 \leq \ell \leq d^+_{G^\rightarrow}(v_i) \leq i$ of out-arcs or along any number $1 \leq \ell \leq i < d^+_{G^\rightarrow}(v_i)$ of out-arcs, with only one predator per out-arc provided that the preyed upon vertex $v_j$ has $j \geq 1$.\\
\noindent 2. Remove the out-arcs along which were predatored and set $d_*^+(v_i) = d^+_{G^\rightarrow}(v_i) - \ell$, and for all vertices $v_{j\neq i}$ which fell prey, set $d_*^- (v_j) = d^-_{G^\rightarrow}(v_j) - 1$.\\
\noindent 3. Set the \emph{predator$_{\geq 0}$-prey$_{\geq 0}$ populations} $\rho_*(v_i) = i-\ell$ and $\rho_*(v_{j \neq i}) = j-1.$\\ 
\noindent 4. Consider the next amended graph $G^\rightarrow_*$ and apply rules 1, 2, 3 and 4 thereto if possible. If not possible, exit.\\ \\ 
\textbf{Observation 1:} We observe that since both the \emph{predator$_{\geq 0}$-prey$_{\geq 0}$ population} of all vertices, and the number of out-arcs embedded in $G^\rightarrow$ as well as those respectively found in the iterative amended graphs $G^\rightarrow_*$ are finite, the Grog algorithm will always terminate. So it can be said informally that the Grog algorithm is well-defined.\\ \\
We note that rule 1 allows us to choose any vertex $v_i$ per iteration. Following that, any number of the existing out-arcs from $v_i$ can be chosen to predator along. Collectively, the specific iterative choices will be called the \emph{predator-prey strategy}. Generally, a number of \emph{predator-prey strategies} may exist for a given $G^\rightarrow$ and the set of all possible strategies is denoted, $S(G^\rightarrow)$. Amongst the strategies there will be those who for a chosen vertex $v_i$, consecutively predator along the maximal number of out-arcs available at $v_i$.  These strategies are called \emph{greedy strategies} and a greedy strategy $s_k \in S(G^\rightarrow)$ is denoted $gs_k$.\\ \\
\textbf{Observation 2:} We observe that if a \emph{predator-prey strategy} $s_k \in S(G^\rightarrow)$ is repeated for a specific graph $G^\rightarrow$, the amended graph $G^\rightarrow_*$ found on termination (exit step) is unique. Put differently, we informally say the \emph{predator-prey strategy} $s_k$ is well-defined.\\ \\
In the final amended graph $G^\rightarrow_*$ (exit step) we will find each vertex $v_i$ has $\rho_{G^\rightarrow_*}(v_i) \geq 0.$ Some residual arcs may be present as well.
\begin{lemma}
In the final amended graph $G^\rightarrow_*$ (exit step) there will be at least one vertex $v_i$ and at least one vertex $v_j$ with $\rho_{G^\rightarrow_*}(v_i) = 0$ and $\rho_{G^\rightarrow_*}(v_j) > 0.$
\end{lemma}
\begin{proof}
\noindent Part 1: Since $G$ is a simple connected graph the open neighborhood of $v_1$ has $N(v_1) \neq \emptyset.$ If in the final amended graph $G^\rightarrow_*$ (exit step), we have $\rho_{G^\rightarrow_*}(v_1) = 0$, Part 1 of the result holds. If $\rho_{G^\rightarrow_*}(v_1) \neq 0$, it implies that $\rho_{G^\rightarrow_*}(v_k) = 0,$ $\forall v_k \in N(v_1).$ Hence, Part 1 of the result holds.\\ \\
\noindent Part 2:  Since $G^\rightarrow$ is a a simple connected directed graph the extremal case is that $v_n$ is preyed upon or predator on, cumulatively over all other $n-1$ vertices of $G^\rightarrow$. Thus we have $\rho_{G^\rightarrow_*}(v_n) = \rho_{G^\rightarrow}(v_n) - (n-1) = n - (n-1) = 1 > 0.$ Hence, Part 2 of the result holds.
\end{proof}
\begin{definition}
For a predator-prey strategy $s_k$ and the final amended graph $G^\rightarrow_*$ (exit step), the cumulative residual, predator$_{\geq 0}$-prey$_{\geq 0}$ population over all vertices is denoted and defined to be $r_{s_k}(G^\rightarrow) = \sum\limits_{\forall v_i}\rho_{G^\rightarrow_*}(v_i).$
\end{definition}
\begin{definition}
The grog number of $G^\rightarrow$ is defined to be $g(G^\rightarrow) = min(r_{s_k}(G^\rightarrow))_{\forall s_k \in S(G^\rightarrow)}$ or equivalently, $g(G^\rightarrow) = min(r_{gs_k}(G^\rightarrow))_{\forall gs_k \in S(G^\rightarrow)}.$ 
\end{definition}
\begin{definition}
The grog number of a simple connected graph $G$ is defined to be $g(G) = min(g(G^\rightarrow))$  over all possible orientations of $G$.
\end{definition}
\noindent Consider a simple connected graph $G$ on $n \geq 2$ vertices with $\epsilon(G)$ edges. It is easy to see that the $n$ vertices can be randomly labelled \emph{(indiced)}, through $v_1, v_2, v_3, ..., v_n$ in $n!$ ways. Equally easy to see that the edges can be orientated in $2^{\epsilon(G)}$ ways. Hence, $\frac{1}{2}n!.2^{\epsilon(G)}$ distinct predator-prey webs can be constructed from $G^\rightarrow.$\\ \\
Let $\Bbb P_{s_k}(G^\rightarrow) =\{(v_i \rightsquigarrow v_j)| v_i$ is predator to $v_j\}.$  Call an arc $(v_i \rightsquigarrow v_j) \in \Bbb P_{s_k}(G^\rightarrow)$ a predator arc. Denote the cardinality of $\Bbb P_{s_k}(G^\rightarrow)$ by c($\Bbb P_{s_k}(G^\rightarrow))$.\\ \\
From the Grog algorithm the interative sequence of $s_k$ can be recorded as an ordered string. So if $s_k$ terminates (exit step) after $t$ iterations we can express $s_k$ as,\\ $s_k =((v_{i_1} \rightsquigarrow v_{j_1}), (v_{i_2} \rightsquigarrow v_{j_2}), (v_{i_3} \rightsquigarrow v_{j_3}), ..., (v_{i_t} \rightsquigarrow v_{j_t})).$ Clearly any pair of predator arcs say, $(v_{i_\ell}, v_{j_\ell})$ and $(v_{i_m}, v_{j_m}),$ $1 \leq \ell, m \leq t$ can interchange positions in the ordered string without changing the value of $r_{s_k}(G^\rightarrow).$ We say that $s_k$ has the commutative property.\\ \\
Clearly, pairs of predator arcs can be grouped together for preferred sequential application prior to other predator arcs, meaning $s_k =((v_{i_1} \rightsquigarrow v_{j_1}), (v_{i_2} \rightsquigarrow v_{j_2}), (v_{i_3} \rightsquigarrow v_{j_3}), ..., (v_{i_t} \rightsquigarrow v_{j_t})) = (((v_{i_s} \rightsquigarrow v_{j_s}), (v_{i_\ell} \rightsquigarrow v_{j_\ell})), (v_{i_w} \rightsquigarrow v_{j_w})_{\forall \emph{other arcs},w \neq s, \ell}).$ We say that $s_k$ has the associative property. The next two lemmas follow.
\begin{lemma}
Consider a specific orientation of a simple connected graph $G$ on $n \geq 2$ vertices labelled $v_1, v_2, ..., v_n$ say, $G^\rightarrow.$ For the initial cumulative predator$_{\geq 0}$-prey$_{\geq 0}$ population given by $\sum\limits_{i=1}^{n}\rho(v_i) = \sum\limits_{i=1}^{n}i$, we have that:
\begin{equation*} 
r_{s_k}(G^\rightarrow)
\begin{cases}
= even, &\text {if and only if, $\sum\limits_{i=1}^{n}i$ is even,}\\ \\ 
= uneven, &\text {if and only if, $\sum\limits_{i=1}^{n}i$ is uneven.}
\end{cases}
\end{equation*} 
\end{lemma}
\begin{proof}
Note that if vertex $v_i$ predator along the arc $(v_i, v_j)$ in step $*$ of the Grog algorithm then $\rho_*(v_i) = \rho_{*-1}(v_i) -1$ and $\rho_*(v_j) = \rho_{*-1}(v_j) -1$ so the total reduction is always 2 for each predator arc in $\Bbb P_{s_k}.$ Hence, $2c(\Bbb P_{s_k}(G^\rightarrow))$ is always even.\\ \\
The two parts now follow immediately from Number Theory.
\end{proof}
\begin{lemma}
For a specific orientation of a simple connected graph $G$ on $n \geq 2$ vertices labelled $v_1, v_2, ..., v_n$ say, $G^\rightarrow$ we have that $c(\Bbb P_{s_k}(G^\rightarrow)) = \frac{1}{2}(\sum\limits_{i=1}^{n}i - r_{s_k}(G^\rightarrow)).$
\end{lemma}
\begin{proof}
From Lemma 2.2 it follows that $r_{s_k}(G^\rightarrow) = \sum\limits_{i=1}^{n}i -2c(\Bbb P_{s_k}(G^\rightarrow)).$ Hence the result:\\ \\
$c(\Bbb P_{s_k}(G^\rightarrow)) = \frac{1}{2}(\sum\limits_{i=1}^{n}i - r_{s_k}(G^\rightarrow)).$
\end{proof}
\subsection{On Paths and Cycles}
\begin{proposition}
If a path $P_n, n\geq 3$ and any specific orientation thereof say, $P_n^\rightarrow$  is extended to $P^\rightarrow_{n+1}$ the residual population over all possible predator-prey strategies applicable to $P_{n+1}^\rightarrow$ is given by:
\begin{equation*} 
r_{s_k^*}(P^\rightarrow_{n+1})
\begin{cases}
= r_{s_k}(P^\rightarrow_n) +(n+1), &\text {if and only if $v_s = v_1$,}\\ \\ 
= r_{s_k}(P^\rightarrow_n) +(n-1), &\text {otherwise,}
\end{cases}
\end{equation*} 
and $s^*_k$ is the minimal deviation from $s_k$ to accommodate arcing to or from $v_{n+1}$ and, $v_{n+1}$ is linked to an end vertex $v_s$ of $P^\rightarrow_n$ or;
\begin{equation*} 
r_{s^*_k}(P^\rightarrow_{n+1})
\begin{cases}
= r_{s_k}(P^\rightarrow_n) + n, &\text {if and only if either $v_p = v_1$ or $v_q = v_1$,}\\ \\ 
= r_{s_k}(P^\rightarrow_n) + (n-1), &\text {otherwise,}
\end{cases}
\end{equation*} 
and $s^*_k$ is the minimal deviation from $s_k$ to accommodate arcing to or from $v_{n+1}$ and, $v_{n+1}$ squeesed inbetween two vertices $v_p, v_q$ of $P^\rightarrow_n$ with $1 \leq p,q \leq n$.
\end{proposition}
\begin{proof}
Consider $P_{n+1}$. We begin by considering any specific orientation of $P_n$  having end vertices $v_s, v_t, 1 \leq s,t \leq n$ and denote it $P^\rightarrow_n$. Clearly the extension from $P^\rightarrow_n$ to $P^\rightarrow_{n+1}$ is made possible by linking (arcing) vertex $v_{n+1}$ to either $v_s$ or $v_t$ or \emph{squeesing} it between two vertices of $P^\rightarrow_n$, say $v_p, v_q$ with $1 \leq p,q \leq n$.\\ \\
Case 1: Assume $v_{n+1}$ is linked to $v_s$. \\ \\
Subcase 1.1: If we consider the arc $(v_{n+1}, v_s)$ in $P^\rightarrow_{n+1}$ the one strategy $r^*_{s_k}$, could be for $v_{n+1}$ to prey on $v_s$ first, leaving a portion of the residual population amounting to $n$ at $v_{n+1}$ and a portion of the residual population amounting to $s-1$ at $v_s$. However if $v_s = v_1$ the vertex $v_1$ cannot predator further on its neighbor in $P_n^\rightarrow$ anymore. So, if $s_k$ is applied from vertex $v_s$ throughout the rest of the remaining path then we have:
\begin{equation*} 
r_{s_k^*}(P^\rightarrow_{n+1})
\begin{cases}
= r_{s_k}(P^\rightarrow_n) +(n+1), &\text {if and only if $v_s = v_1$,}\\ \\ 
= r_{s_k}(P^\rightarrow_n) + (n-1), &\text {otherwise.}
\end{cases}
\end{equation*} 
If we consider the strategy to apply $s_k$ effecting from $v_s$ first, the vertex $v_{n+1}$ cannot predator at all if $v_s = v_1$. If $v_s \neq v_1$ there is a loss of 2 at $v_s$ and a loss of 1 at $v_{n+1}$ from the initial total  \emph{predator$_{\geq 0}$-prey$_{\geq 0}$ population} in $P^\rightarrow_{n+1}$ . Hence the result:
\begin{equation*} 
r_{s^*_k}(P^\rightarrow_{n+1})
\begin{cases}
= r_{s_k}(P^\rightarrow_n) + (n+1), &\text {if and only if $v_s = v_1$,}\\ \\ 
= r_{s_k}(P^\rightarrow_n) + (n-1), &\text {otherwise,}
\end{cases}
\end{equation*}
holds.\\ \\ 
Subcase 1.2: If we consider the arc $(v_s, v_{n+1})$ in $P_{n+1}^\rightarrow$ the one strategy $r^*_{s_k}$, could be for $v_s$ to prey on $v_{n+1}$ first, leaving a portion of the residual population amounting to $n$ at $v_{n+1}$ and a portion of the residual population amounting to $s-1$ at $v_s$. Now, if $s_k$ is applied from vertex $v_s$ throughout the rest of the remaining path then we have:
\begin{equation*} 
r_{s_k^*}(P^\rightarrow_{n+1})
\begin{cases}
= r_{s_k}(P^\rightarrow_n) +(n+1), &\text {if and only if $v_s = v_1$,}\\ \\ 
= r_{s_k}(P^\rightarrow_n) + (n-1), &\text {otherwise.}
\end{cases}
\end{equation*} 
If we consider the strategy to apply $s_k$ effecting from $v_s$ first, the vertex $v_{n+1}$ cannot predator at all if $v_s = v_1$. If $v_s \neq v_1$  there is a loss of 2 at $v_s$ and a loss of 1 at $v_{n+1}$ from the initial total  \emph{predator$_{\geq 0}$-prey$_{\geq 0}$ population} in $P^\rightarrow_{n+1}$ . Hence the result:
\begin{equation*} 
r_{s^*_k}(P^\rightarrow_{n+1})
\begin{cases}
= r_{s_k}(P^\rightarrow_n) + (n+1), &\text {if and only if $v_s = v_1$,}\\ \\ 
= r_{s_k}(P^\rightarrow_n) + (n-1), &\text {otherwise,}
\end{cases}
\end{equation*} 
holds.\\ \\
Case 2: Assume $v_{n+1}$ is linked to $v_t$. \\ \\
The proof of this case follows similar to that of Case 1.\\ \\
Case 3: Assume $v_{n+1}$ is squeesed inbetween vertices $v_p, v_q$ with $1 \leq p, q \leq n$.\\ \\
Subcase 3.1: The arcs $(v_p, v_{n+1})$ and $(v_{n+1}, v_q)$ exist in $P^\rightarrow_{n+1}.$ The one strategy $r^*_{s_k}$  in $P^\rightarrow_{n+1}$, could be for $v_p$ to prey on $v_{n+1}$ first, leaving a portion of the residual population amounting to $n$ at $v_{n+1}$ and a portion of the residual population amounting to $p-1$ at $v_p$. Then let $v_{n+1}$ prey on $v_q$, leaving a portion of the residual population amounting to $n-1$ at $v_{n+1}$ and a portion of the residual population amounting to $q-1$ at $v_q$. Now, if $s_k$ is applied throughout the rest of the remaining path then we have:\\ \\
\begin{equation*} 
r_{s^*_k}(P^\rightarrow_{n+1})
\begin{cases}
= r_{s_k}(P^\rightarrow_n) + n, &\text {if and only if either $v_p = v_1$ or $v_q = v_1$,}\\ \\ 
= r_{s_k}(P^\rightarrow_n) + (n-1), &\text {otherwise,}
\end{cases}
\end{equation*} 
holds.\\ \\
Subcase 3.2: The arcs $(v_p, v_{n+1})$ and $(v_q, v_{n+1})$ exist in $P^\rightarrow_{n+1}.$ By applying the strategy $s^*_k$ to accommodate vertex $v_{n+1}$ and with similar reasoning as in Subcase 3.1, the result:
\begin{equation*} 
r_{s^*_k}(P^\rightarrow_{n+1})
\begin{cases}
= r_{s_k}(P^\rightarrow_n) + n, &\text {if and only if either $v_p = v_1$ or $v_q = v_1$,}\\ \\ 
= r_{s_k}(P^\rightarrow_n) + (n-1), &\text {otherwise,}
\end{cases}
\end{equation*} 
holds.\\ \\
Subcase 3.3: The arcs $(v_{n+1}, v_p)$ and $(v_{n+1}, v_q)$ exist in $P^\rightarrow_{n+1}.$ By applying the strategy $s^*_k$ to accommodate vertex $v_{n+1}$ and with similar reasoning as in Subcase 3.1, the result:
\begin{equation*} 
r_{s^*_k}(P^\rightarrow_{n+1})
\begin{cases}
= r_{s_k}(P^\rightarrow_n) + n, &\text {if and only if either $v_p = v_1$ or $v_q = v_1$,}\\ \\ 
= r_{s_k}(P^\rightarrow_n) + (n-1), &\text {otherwise,}
\end{cases}
\end{equation*} 
holds.\\ \\
Subcase 3.4: The arcs $(v_q, v_{n+1})$ and $(v_{n+1}, v_p)$ exist in $P^\rightarrow_{n+1}.$ By applying the strategy $s^*_k$ to accommodate vertex $v_{n+1}$ and with similar reasoning as in Subcase 3.1, the result:
\begin{equation*} 
r_{s^*_k}(P^\rightarrow_{n+1})
\begin{cases}
= r_{s_k}(P^\rightarrow_n) + n, &\text {if and only if either $v_s = v_1$ or $v_q = v_1$,}\\ \\ 
= r_{s_k}(P^\rightarrow_n) + (n-1), &\text {otherwise,}
\end{cases}
\end{equation*} 
holds.
\end{proof}
\begin{corollary}
For a path $P_n, n\geq 3$ we have that $g(P_{n+1}) = g(P_n) + (n-1).$
\end{corollary}
\begin{proof}
From Proposition 2.4 it follows that for a specific orientation of $P_n$ and for all possible predator-prey strategies,\\ \\
$g(P_{n+1}^\rightarrow) = min\{r_{s_k}(P^\rightarrow_n) + (n+1), r_{s_k}(P^\rightarrow_n) + n, r_{s_k}(P^\rightarrow_n) + (n-1)\}_{\forall s_k \in S(P^\rightarrow_n)} =\\ \\ min\{r_{s_k}(P^\rightarrow_n)\}_{\forall s_k \in S(P^\rightarrow_n)} + (n-1).$\\ \\
From definition 2.3 it then follows that $g(P_{n+1}) = g(P_n) + (n-1).$
\end{proof}
\noindent \textbf{Example 1.} Consider path $P_3$. Note vertex labelling will be from \emph{left to right}. Also note that the labelled vertices will be denoted through an \emph{ordered triplet} and the arcs through an \emph{ordered airc-pair}. The $\frac{1}{2}.3!.2^2 = 12$ distinct predator-prey webs are:\\ \\
\noindent (1) $ V(P_3^\rightarrow) = \{v_1, v_2, v_3\}$ and $A(P_3^\rightarrow) = \{(v_1, v_2), (v_2, v_3)\}$\\
\noindent (2) $ V(P_3^\rightarrow) = \{v_1, v_2, v_3\}$ and $A(P_3^\rightarrow) = \{(v_1, v_2), (v_3, v_2)\}$\\
\noindent (3) $ V(P_3^\rightarrow) = \{v_1, v_2, v_3\}$ and $A(P_3^\rightarrow) = \{(v_2, v_1), (v_2, v_3)\}$\\
\noindent (4) $ V(P_3^\rightarrow) = \{v_1, v_3, v_2\}$ and $A(P_3^\rightarrow) = \{(v_1, v_3), (v_3, v_2)\}$\\
\noindent (5) $ V(P_3^\rightarrow) = \{v_1, v_3, v_2\}$ and $A(P_3^\rightarrow) = \{(v_1, v_3), (v_2, v_3)\}$\\
\noindent (6) $ V(P_3^\rightarrow) = \{v_1, v_3, v_2\}$ and $A(P_3^\rightarrow) = \{(v_3, v_1), (v_3, v_2)\}$\\
\noindent (7) $ V(P_3^\rightarrow) = \{v_2, v_1, v_3\}$ and $A(P_3^\rightarrow) = \{(v_2, v_1), (v_1, v_3)\}$\\
\noindent (8) $ V(P_3^\rightarrow) = \{v_2, v_1, v_3\}$ and $A(P_3^\rightarrow) = \{(v_2, v_1), (v_3, v_1)\}$\\
\noindent (9) $ V(P_3^\rightarrow) = \{v_2, v_1, v_3\}$ and $A(P_3^\rightarrow) = \{(v_1, v_2), (v_1, v_3)\}$\\
\noindent (10) $ V(P_3^\rightarrow) = \{v_2, v_3, v_1\}$ and $A(P_3^\rightarrow) = \{(v_2, v_3), (v_3, v_1)\}$\\
\noindent (11) $ V(P_3^\rightarrow) = \{v_3, v_1, v_2\}$ and $A(P_3^\rightarrow) = \{(v_3, v_1), (v_1, v_2)\}$\\
\noindent (12) $ V(P_3^\rightarrow) = \{v_3, v_2, v_1\}$ and $A(P_3^\rightarrow) = \{(v_3, v_2), (v_2, v_1)\}$\\ \\
For each case the number of greedy strategies together with the residual population $r_{s_k}(P^\rightarrow_3)$ as well as the grog number $g(P^\rightarrow_3)$ will be depicted.\\ \\
\noindent (1) Number of greedy strategies = 2; $r_{s_1}(P^\rightarrow_3) = 2,$ $r_{s_2}(P^\rightarrow_3) = 2$ and $g(P^\rightarrow_3) = min\{2,2\} = 2.$\\ 
\noindent (2) Number of greedy strategies = 2; $r_{s_1}(P^\rightarrow_3) = 2,$ $r_{s_2}(P^\rightarrow_3) = 2$ and $g(P^\rightarrow_3) = min\{2,2\} = 2.$\\ 
\noindent (3) Number of greedy strategies = 2; $r_{s_1}(P^\rightarrow_3) = 2$ $r_{s_2}(P^\rightarrow_3) = 2$ and $g(P^\rightarrow_3) = min\{2, 2\} = 2.$\\ 
\noindent (4) Number of greedy strategies = 2; $r_{s_1}(P^\rightarrow_3) = 2,$ $r_{s_2}(P^\rightarrow_3) = 2$ and $g(P^\rightarrow_3) = min\{2,2\} = 2.$\\ 
\noindent (5) Number of greedy strategies = 2; $r_{s_1}(P^\rightarrow_3) = 2,$ $r_{s_2}(P^\rightarrow_3) = 2$ and $g(P^\rightarrow_3) = min\{2,2\} = 2.$\\ 
\noindent (6) Number of greedy strategies = 2; $r_{s_1}(P^\rightarrow_3) = 2,$ $r_{s_2}(P^\rightarrow_3) = 2$ and $g(P^\rightarrow_3) = min\{2, 2\} = 2.$\\ 
\noindent (7) Number of greedy strategies = 2; $r_{s_1}(P^\rightarrow_3) = 4,$ $r_{s_2}(P^\rightarrow_3) = 4$ and $g(P^\rightarrow_3) = min\{4,4\} = 4.$\\ 
\noindent (8) Number of greedy strategies = 2; $r_{s_1}(P^\rightarrow_3) = 4,$ $r_{s_2}(P^\rightarrow_3) = 4$ and $g(P^\rightarrow_3) = min\{4,4\} = 4.$\\ 
\noindent (9) Number of greedy strategies = 2; $r_{s_1}(P^\rightarrow_3) = 4,$ $r_{s_2}(P^\rightarrow_3) = 4$ and $g(P^\rightarrow_3) = min\{4,4\} = 4.$\\ 
\noindent (10) Number of greedy strategies = 2; $r_{s_1}(P^\rightarrow_3) = 2,$ $r_{s_2}(P^\rightarrow_3) = 2$ and\\ $g(P^\rightarrow_3) = min\{2,2\} = 2.$\\
\noindent (11) Number of greedy strategies = 2; $r_{s_1}(P^\rightarrow_3) = 4,$ $r_{s_2}(P^\rightarrow_3) = 4$ and\\ $g(P^\rightarrow_3) = min\{4,4\} = 4.$\\ 
\noindent (12) Number of greedy strategies = 2; $r_{s_1}(P^\rightarrow_3) = 2,$ $r_{s_2}(P^\rightarrow_3) = 2$ and\\ $g(P^\rightarrow_3) = min\{2,2\} = 2.$\\  \\
From definition 2.3 it follows that $g(P_3) = min \{2, 4\} = 2.$
\begin{theorem}
For all simple connected graphs on $n \geq 3, n \in \Bbb N$ vertices, over all indices and over all orientations of $G,$  there exists at least one indice with at least one orientation say orientation $o_i$ with corresponding directed graph $G^\rightarrow_{o_i}$ and at least another indice with at least one orientation say orientation $o_j$ with corresponding directed graph $G^\rightarrow_{o_j}$, $o_i \neq o_j$ with $g(G^\rightarrow_{o_i}) \neq g(G^\rightarrow_{o_j}).$
\end{theorem}
\begin{proof}
In any simple connected graph $G$ on $ n \geq 3, n \in \Bbb N$, at least one induced subgraph $\langle G\rangle_i$ on $i\leq n$ vertices exists with $\langle G \rangle_i  \simeq P_3$. Hence in any simple connected graph $G$ on $n \geq 3$, $n \in \Bbb N$ vertices we can find as a minimal indiced case, the subgraph $P_3$ with $V(P_3) = \{v_1, v_2, v_3\}.$\\ \\
Consider the orientation $o_1 = \{(v_1, v_2), (v_2, v_3)\}.$ So for $P^\rightarrow_{3, o_1}$ we have $r_{s_1}(P^\rightarrow_{3,o_1}) = 2,$ $r_{s_2}(P^\rightarrow_{3,o_1}) = 2$ and $g(P^\rightarrow_{3,o_1}) = min\{2,2\} = 2.$\\ \\
Now consider the orientation $o_2 = \{(v_2, v_1), (v_1, v_3)\}.$ So for $P^\rightarrow_{3, o_2}$ we have $r_{s_1}(P^\rightarrow_{3,o_2}) = 4,$ $r_{s_2}(P^\rightarrow_{3,o_2}) = 4$ and $g(P^\rightarrow_{3,o_2}) = min\{4,4\} = 4.$\\ \\
So we have $g(P^\rightarrow_{3,o_1}) \neq g(P^\rightarrow_{3,o_2})$.\\ \\
Considering the partial graph $G - P_3$, let $g((G - P_3)^\rightarrow_{o_i} = t$ for a specific orientation $o_i$ and a \emph{predator-prey strategy} $s_k$.\\ \\
For the path $P_3 = v_1v_2v_3$ consider the graph $G^\rightarrow$ with the \emph{predator-prey strategy} $s^*_k = (s_1, s_k)$ and the orientation $o^*_i = \{(v_1,v_2), (v_2, v_3), o_i\}.$ Clearly arcs to and from $P_3^\rightarrow$ and the rest of $G^\rightarrow$ may exist such that upon applying $s^*_k$ we have $g((G - P_3)^\rightarrow_{o_i}) - 4 = t - 4 \leq g(G^\rightarrow)_{o^*_i} \leq t+2 = g((G - P_3)^\rightarrow_{o_i}) + 2.$\\ \\
For the path $P_3 = v_2v_1v_3$ consider the graph $G^\rightarrow$ with the \emph{predator-prey strategy} $s^*_k = (s_1, s_k)$ and the orientation $o^*_j = \{(v_2,v_1), (v_1, v_3), o_j\}.$ Clearly arcs to and from $P^\rightarrow_3$ and the rest of $G^\rightarrow$ may exist such that upon applying $s^*_k$ we have $g((G - P_3)^\rightarrow_{o_j}) - 8 = t - 8 \leq g(G^\rightarrow)_{o^*_j} \leq t+4 = g((G - P_3)^\rightarrow_{o_j}) + 4.$\\ \\
Since the set of arcs between the two directed paths and the rest of $G^\rightarrow$ must specifically remain the same, it follows that:\\ \\
$g((G - P_3)^\rightarrow_{o_i}) - 4 = t - 4 \leq g(G^\rightarrow)_{o^*_i} \leq t+2 = g((G - P_3)^\rightarrow_{o_i}) + 2 \neq \\\\
g((G - P_3)^\rightarrow_{o_j}) - 8 = t - 8 \leq g(G^\rightarrow)_{o^*_j} \leq t+4 = g((G - P_3)^\rightarrow_{o_j}) + 4.$\\ \\
Therefore, in general the result that there exist at least two different orientations and at least two different indices of $G$, such that $g(G^\rightarrow_{o_i}) \neq g(G^\rightarrow_{o_j})$, follows.
\end{proof}
\begin{proposition}
If a cycle $C_n, n\geq 3$ and any specific orientation thereof say, $C_n^\rightarrow$  is extended to $C^\rightarrow_{n+1}$ the residual population over all possible predator-prey strategies applicable to $C_{n+1}^\rightarrow$ is given by:\\ \\
$r_{s^*_k}(C_{n+1}^\rightarrow) = r_{s_k}(C_n^\rightarrow) + (n-1),$\\ \\
and $s^*_k$ is the minimal deviation from $s_k$ to accommodate arcing to or from $v_{n+1}$.
\end{proposition}
\begin{proof}
Consider $C_{n+1}$. We begin by considering any specific orientation of $C_n$ and denote it $C^\rightarrow_n$. Clearly the extension from $C^\rightarrow_n$ to $C^\rightarrow_{n+1}$ is made possible by \emph{squeesing} $v_{n+1}$, (arcing) inbetween two vertices of $C^\rightarrow_n$, say $v_s, v_t$ with $1 \leq s,t \leq n$.\\ \\
Without loss of generality consider the arc $(v_s, v_t)$ in $C_n^\rightarrow$. Begin by letting $v_s$ prey on $v_t$ by predatoring along the arc $(v_s, v_t)$. After this first step of the Grog algorithm the rest of the application applies to a path $P_{n*}^\rightarrow$ with end vertices $v_s$ and $v_t$ having $\rho(v_s) = s-1$ and $\rho(v_t) = t-1$. After applying $s_k$ to this path the value $r_{s_k}(C^\rightarrow_n)$ is obtained.\\ \\
We now squeese $v_{n+1}$ inbetween $v_s, v_t$ and for any one of the four possible orientation between $v_s, v_{n+1}, v_t$ we have that if the Grog algorithm is applied to path $P_3^\rightarrow$, $V(P_3) = \{v_s, v_{n+1}, v_t\}$ we obtain $\rho(v_s) = s-1$, $\rho(v_{n+1}) = n-1$, $\rho(v_t) = t-1.$\\ \\
Furthering with the Grog algorithm we are left with exactly the path $P_{n*}^\rightarrow$ mentioned above. Hence the result:\\ \\
$r_{s^*_k}(C_{n+1}^\rightarrow) = r_{s_k}(C_n^\rightarrow) + (n-1),$\\ \\
and $s^*_k$ is the minimal deviation from $s_k$ to accommodate $v_{n+1}$, follows.
\end{proof}
\begin{corollary}
For a cycle $C_n, n \geq 3$ we have that $r_{s_k^*}(C_n^\rightarrow) = r_{s_k}(P_n^\rightarrow) -2$, with $s^*_k$ the minimal deviation from $s_k$ to accommodate an orientation of the edge $v_pv_q$ with $v_p$ and $v_q$ the end vertices of $P_n^\rightarrow$.
\end{corollary}
\begin{proof}
The result follows directly from the proof of Proposition 2.7.
\end{proof}
\subsection{On Jaco graphs, $J_n(1), n \in \Bbb N$, $n \geq 2$}
As stated in Kok et. al. [3], finding a closed formula for the number of edges of a finite Jaco graph will assist in finding closed formulae for many recursive results found for Jaco graphs. In the absence of such formula we present the next proposition. We begin with a lemma.
\begin{lemma}
For a Jaco graph, $J_n(1)$, $n \geq 2$ having the Jaconian vertex $v_i$ we have that\\ $2i - n \geq 0.$
\end{lemma}
\begin{proof}
From the definition of a Jaco graph $J_n(1)$, $n\geq 2$ with Jaconian vertex $v_i$ we have that either $i + d^+(v_i) = n$ or $i + d^+(v_i) = n-1.$ Therefore:\\ \\
Case 1: $i + d^+(v_i) = n\\ \\
\therefore i = n - d^+(v_i),\\
\therefore 2i = 2n - 2d^+(v_i),\\
\therefore 2i -n = 2n -2d^+(v_i) - n = n - 2d^+(v_i) = (n - d^+(v_i)) - d^+(v_i) = i - d^+(v_i).$\\ \\
Since $i - d^+(v_i) = d^-(v_i)$ and $d^-(v_i) \geq 0$ in $J_n(1)$, $n \geq 2$ the result follows.\\ \\
Case 2: $i + d^+(v_i) = n-1\\ \\
\therefore i = (n-1) - d^+(v_i),\\
\therefore 2i = 2(n-1) - 2d^+(v_i),\\
\therefore 2i -n = 2(n-1) -2d^+(v_i) - n = n - 2d^+(v_i) -2 = ((n-1) - d^+(v_i)) - d^+(v_i) - 1\\= (i - d^+(v_i)) - 1= d^-(v_i) -1.$\\ \\
Since $d^-(v_i) \geq 1$ in $J_n(1)$, $n \geq 2$ the result follows.
\end{proof}
\begin{proposition}
For a Jaco graph, $J_n(1)$, $n \geq 2$ having the Jaconian vertex $v_i$ we have that:\\ \\
$g(J_{n+1}(1)) = g(J_n(1)) + (2i - n) + 1.$
\end{proposition}
\begin{proof}
Consider any Jaco graph $J_n(1)$, $n \geq 2$ with Jaconian vertex $v_i$ and grog number $g(J_n(1)).$ In extending to $J_{n+1}(1)$ the vertex $v_{n+1}$ with arcs $(v_{i+1}, v_{n+1}), (v_{i+2}, v_{n+1}), ..., (v_n, v_{n+1})$ are added to $J_n(1).$ So clearly $n-i$ additional arcs are added.\\ \\
By applying the Grog algorithm to vertices $v_j$, $i+1 \leq j \leq n$ along the respective arcs $(v_j, v_{n+1})$, $i+1 \leq j \leq n$ there is a corresponding cumulative reduction in the residual population at vertices $v_j$, $i+1 \leq j \leq n$ of $n - i$. Furthermore, there is an corresponding increase in the residual population at vertex $v_{n+1}$ of $(n+1) -(n - i)$. Hence,\\ \\
$g(J_{n+1}(1)) = g(J_n(1)) - (n-i) + ((n + 1) - (n - i)) = g(J_n(1)) + (2i - n) + 1.$
\end{proof}
\begin{corollary}
For a Jaco graph, $J_n(1)$, $n \geq 2$ we have $g(J_{n+1}(1)) > g(J_n(1)).$
\end{corollary}
\begin{proof} 
Since $(2i - n) \geq 0$, $n \geq 3$ (Lemma 2.9) and $g(J_n(1)) + 1 > g(J_n(1))$, $n \geq 2$ it follows that $g(J_n(1)) + (2i - n) + 1 > g(J_n(1)),$ $n \geq 2$. Therefore, $g(J_{n+1}(1)) > g(J_n(1)),$ $n \geq 2$. 
\end{proof}
\noindent [Open problem 1: Formalise Observation 1, mathematically.] \\ 
\noindent [Open problem 2: Formalise Observation 2, mathematically.] \\ 
\noindent [Open problem 3: The Grog algorithm has been described informally. Formalise the Grog algorithm.]\\
\noindent [Open problem 4: For a given $G^\rightarrow$ find the number of possible predator-prey strategies, (cardinality of $S(G^\rightarrow))$ and if possible describe the algorithmic efficiency of determining it.] \\  \\
\textbf{\emph{Open access:}} This paper is distributed under the terms of the Creative Commons Attribution License which permits any use, distribution and reproduction in any medium, provided the original author(s) and the source are credited. \\ \\
References (Limited) \\ \\
$[1]$  Bondy, J.A., Murty, U.S.R., \emph {Graph Theory with Applications,} Macmillan Press, London, (1976). \\
$[2]$ Cohen, J. E., \emph{Interval graphs and food webs; A finding and a problem.}, Document 17696-PR, RAND Corporation, 1968. \\
$[3]$ Kok, J., Fisher, P., Wilkens, B., Mabula, M., Mukungunugwa, V., \emph{Characteristics of Finite Jaco Graphs, $J_n(1), n \in \Bbb N$}, arXiv: 1404.0484v1 [math.CO], 2 April 2014. \\
$[4]$ Merz, S.K., \emph{Competition Graphs, p-Competition Graphs, Two-Step Graphs, Squares, and Domination Graphs.}, Ph.D. thesis, University of Colorado at Denver, Department of Applied Mathematics, 1995.\\
$[5]$ Rasmussen, C.W., \emph{Interval Competition Graphs of Symmetric D-graphs and Two-Step Graphs of Trees.}, Ph.D. thesis, University of Colorado at Denver, Department of Mathematics, 1990.\\
$[6]$ Raychaudhuri, A., \emph{Intersection Assignments, T-colorings, and Powers of Graphs.}, Ph.D. thesis, Rutgers University, 1987.
\end{document}